\documentclass[12pt]{amsart}
\usepackage{amssymb}
\usepackage[headings]{fullpage}
\usepackage{amsfonts}
\usepackage{extarrows}
\usepackage{paralist}
\usepackage{xspace}
\usepackage{lscape}
\usepackage{euscript}
\usepackage{color,tikz}
\usetikzlibrary{decorations.pathreplacing,calc}
\usepackage{epigraph}
\usepackage[all]{xy}
\usepackage[normalem]{ulem}
\usepackage[colorlinks,pagebackref=true,pdftex]{hyperref}
\usepackage[alphabetic,backrefs,abbrev,lite]{amsrefs} 

\makeatletter
\@addtoreset{equation}{section}
\def\theequation{\thesection.\@arabic \c@equation}

\def\theenumi{\@roman\c@enumi}
\def\@citecolor{blue}
\def\@linkcolor{blue}
\def\@urlcolor{blue}

\makeatother

\newtheorem{lemma}[equation]{Lemma}
\newtheorem{prop}[equation]{Proposition}

\newtheorem{conj}[equation]{Conjecture}

\newtheorem{claim*}{Claim}
\newtheorem{thm}[equation]{Theorem}

\newtheorem{sufficient condition}[equation]{Claim}

\theoremstyle{definition}
\newtheorem{rmk}[equation]{Remark}
\newenvironment{remark}[1][]{%
    \begin{rmk}[#1]}{\end{rmk}}

\newtheorem{eg}[equation]{Example}
\newenvironment{example}[1][]{%
    \begin{eg}[#1] }{\end{eg}}

\newtheorem{definition}[equation]{Definition}

\newtheorem{notn}[equation]{Notation}

\def\<{\langle}
\def\>{\rangle}

\newcommand{\coker}{\operatorname{coker}}

\newcommand{\Hom}{\operatorname{Hom}} 

\newcommand{\im}{\operatorname{im}}

\newcommand{\Proj}{\operatorname{Proj}}

\newcommand{\Spec}{\operatorname{Spec}}

\newcommand{\Tor}{\operatorname{Tor}}

\renewcommand{\to}{\longrightarrow}
\newcommand{\from}{\longleftarrow}
\newcommand{\kk}{\Bbbk}

\newcommand{\FF}{\mathbf{F}}

\newcommand{\GG}{\mathbf{G}}
\newcommand{\HH}{\mathrm{H}}

\newcommand{\KK}{\mathbf{K}}

\newcommand{\cO}{{\mathcal{O}}}
\newcommand{\PP}{\mathbb{P}}
\newcommand{\QQ}{\mathbb{Q}}

\newcommand{\VV}{\mathbb{V}}

\newcommand{\ZZ}{\mathbb{Z}}

\newcommand{\defi}[1]{{\bfseries\upshape #1}}

\newcommand{\cF}{\mathcal F}
\newcommand{\cG}{\mathcal G}
\newcommand{\cL}{\mathcal L}

\title{Three flavors of Extremal Betti tables}
\author[C. Berkesch]{Christine Berkesch}
\address{Department of Mathematics \\ Duke University, Box 90320 \\
  Durham, NC 27708} 
\email{cberkesc@math.duke.edu}

\author[D. Erman]{Daniel Erman}
\address{Department of Mathematics \\ University of Michigan \\
Ann Arbor, MI, 48109}
\email{derman@umich.edu}

\author[M. Kummini]{Manoj Kummini}
\address{Chennai Mathematical Institute \\ 
    Siruseri, Tamilnadu, 603103. India.}
\email{mkummini@cmi.ac.in}

\thanks{The first author was partially supported by NSF Grant DMS 0901123. 
The second author was supported by NSF Award No.~1003997 and by a Simons Foundation fellowship.}

\begin{document}
\begin{abstract}
We discuss extremal Betti tables of resolutions in three different
contexts.  We begin over the graded polynomial ring, where extremal Betti
tables correspond to pure resolutions.  We then contrast this behavior with
that of extremal Betti tables over regular local rings and over a bigraded
ring.
\end{abstract}

\maketitle
\section{Introduction}
\label{sec:intro}
Classification problems can be often discretized by replacing a collection
of complicated objects by numerical invariants. For instance, if we are
interested in modules over a local or graded ring, then we can study their
Hilbert Polynomial, Betti numbers, Bass numbers, and more.  Describing the behavior
of these invariants becomes a proxy for understanding the modules;
identifying the extremal behavior of an invariant provides structural limitations.

The conjectures of M.~Boij and J.~S\"oderberg~\cite{boij-sod1}, proven
by D.~Eisenbud and F.O.\ Schreyer \cite{eis-schrey1}, link the
extremal properties of invariants of free resolutions over the graded
polynomial ring $S=\kk[x_1, \dots, x_n]$ with the
Herzog--Huneke--Srinivasan Multiplicity Conjectures. Here
$\kk$ is any field, $S$ has the standard $\ZZ$-grading, and we study the
graded Betti tables of $S$-modules.  The Boij--S\"oderberg Conjectures state that
the extremal rays of the cone of Betti tables are given by Betti
tables of Cohen--Macaulay modules with pure resolutions.
There exist two excellent introductions to
Boij--S\"oderberg Theory~\cites{eis-schrey-icm,floy-expository}.

In this paper, we explore the notion of an extremal Betti
table in three different contexts: in the original
setting of a standard graded polynomial ring; over a regular local ring;
and over a finely graded polynomial ring.

Previous work has considered the extremal behavior of free resolutions,
in a manner unconnected to Boij--S\"oderberg theory.
Each graded Betti number of the Eliahou--Kervaire resolution of
a lex-segment ideal is known to be maximal among cyclic modules with the
same Hilbert function~\cites{bigatti, hulett, pardue}.
Also,~\cite{avramov-extremal} studies the Betti numbers of modules with
extremal homological dimensions, complexity, or curvature.  Though we will
not discuss these types of results further, the interested reader might
consider~\cites{avramov-infinite,peeva-book,iyengar-pardue}.

Throughout this paper, $S$ will denote a standard graded polynomial ring,
$R$ will denote a regular local ring, and $T$ will denote a finely graded
polynomial ring. For a graded $S$-module $M$, we define the \defi{graded
Betti numbers}
$\beta_{i,j}(M) := \dim_\kk \Tor^S_i(M,\kk)_j$. 
Betti numbers also have a more concrete interpretation: if  
$\FF=[\FF_0\gets\FF_1\gets \dots \gets \FF_n\gets 0]$ 
is a minimal graded free resolution of $M$, then $\beta_{i,j}(M)$ is the
number of minimal generators of $\FF_i$ of degree $j$. The \defi{graded
Betti table} of $M$, denoted $\beta(M)$, is the vector with coordinates
$\beta_{i,j}M$ in the vector space $\VV = \bigoplus_{i=0}^n \bigoplus_{j\in
\ZZ}\QQ$.

For local and multigraded rings, there are analogous definitions.  For a
regular local ring $R$ with residue field $\kk$, we  define the \defi{(local)
Betti numbers} of an $R$-module as $\beta^R_{i}(M)=\dim_{\kk}
\Tor^R_i(M,\kk)$.  Over a $\ZZ^m$-graded polynomial ring $T$, we define the
\defi{multigraded Betti numbers} of a $T$-module $M$ as
$\beta^T_{i,\alpha}(M)=\dim_{\kk} \Tor^T_i(M,\kk)_{\alpha}$, where
$\alpha\in \ZZ^m$. We denote the respective Betti tables by $\beta^R(M)$
and $\beta^T(M)$.

To streamline the exposition, we focus on modules of finite length. With
minor adjustments, most results we discuss can be extended to the case of
finitely generated modules.
See~\cites{boij-sod2,eis-schrey-icm,floy-expository} for the standard
graded case and \cite{beks-local} for the local case.
 
Let $M$ be a graded $S$-module (or an $R$-module or a multigraded
$T$-module) of finite length.  We say that $\beta(M)$ is \defi{extremal}
if, for any decomposition of the form
\[
\beta(M)=\beta(M')+\beta(M'')
\]
with $M', M''$ graded $S$-modules (or $R$-modules or multigraded
$T$-modules, respectively), we have that $\beta(M')$ is a scalar multiple
of $\beta(M)$.
Extremal Betti tables correspond to extremal rays of the cone of Betti
tables of finite length.  In the case of $S$, this is the cone
\[
B_\QQ^\text{fin}(S) := \QQ_{\geq 0}\cdot \{ \beta(M) \mid M \text{ is a
graded $S$-module of finite length}\} \subseteq \VV.
\]
Boij and S\"oderberg observed that for graded $S$-modules, there
is a natural sufficient condition for extremality.
\begin{sufficient condition}\label{cond:1}
For a graded $S$-module $M$ of finite length, if $M$ has a pure resolution,
then $\beta(M)$ is extremal.
\end{sufficient condition}
\noindent Here we say that $M$ has a \defi{pure resolution} if, for each
$i$, $\beta_{i,j}(M)\ne 0$ for at most one $j$. After proving the claim,
Boij and S\"oderberg conjectured that this condition is not only sufficient
but also necessary. In fact, after imposing some obvious degree
restrictions on the Betti table, they conjecture the existence of pure
resolutions of Cohen--Macaulay modules of essentially any combinatorial
type. This was later proven by \cite{efw} in characteristic $0$ and
by~\cite{eis-schrey1} in a characteristic-free manner; see
Theorem~\ref{thm:exist std}.

In \S\ref{sec:graded}, we first quickly review why Claim~\ref{cond:1}
provides a sufficient condition for extremality.  The remainder of the
section is an expository overview of Eisenbud and Schreyer's construction
of modules with pure resolutions.  

We then turn our attention to the case of a regular local ring, as
considered in~\cite{beks-local}.  In contrast with the graded case, there
is no obvious analogue of Claim~\ref{cond:1}.  In retrospect this is
inevitable, as there are no modules of finite length whose Betti tables are
extremal.

In the final section, we move in the opposite direction, refining the
grading to a finely graded polynomial ring $T$.  
One possibility for understanding extremal Betti tables in the multigraded setting is to seek out multigraded lifts of
pure resolutions from the standard $\ZZ$-graded setting.  This approach is taken in \cite{floystad-multigraded}, which considers the linear space of such multigraded Betti tables.  Moreover, in the case of $\kk[x,y]$ with $\ZZ^2$-grading, \cite{boij-floystad} constructs the entire cone of bigraded Betti tables spanned by such lifted pure resolutions.

Not all extremal Betti tables arise in this way in the multigraded setting, and we provide a
sufficient condition for a bigraded Betti table to be extremal, which demonstrates this fact.  The extra
rigidity induced by the bigrading seems to greatly complicate the picture.
We use this condition to show the existence of a zoo of
extremal Betti tables.

\section*{Acknowledgments}
We thank David Eisenbud and Frank-Olaf Schreyer
for conversations which shaped and strengthened many of these ideas, particularly
those  in~\S\ref{sec:multi}.
We came to understand the details of the proof of Eisenbud and Schreyer during a Mathematical Research
Communities workshop at Snowbird, Utah, organized by Mike Stillman and Hal
Schenck; we thank them and the other participants for their help. We thank
the National Science Foundation which supported this AMS MRC workshop
through grant DMS-0751449.  We thank Steven Sam for helpful ideas.  
\section{Preliminaries}
\label{sec:background}
Given a ring $R$ (or a scheme $X$) and a complex $\FF$ of $R$-modules (or
$\cO_X$-modules) with differential $\partial_i: \FF_i\to \FF_{i-1}$, we
denote the homology modules of $\FF$ by $\HH_i(\FF)=(\ker \partial_i)/(\im
\partial_{i+1})$. The derived category of $R$-modules (or of
$\cO_X$-modules) is the category consisting complexes of
$R$-modules (or $\cO_X$-modules) modulo the equivalence relation generated
by quasi-isomorphisms. We may represent any object in the derived category
by a genuine complex of modules.

For a projection of the form $\pi_1 \colon X\times \PP^m\to X$ of schemes,
there are well-defined higher direct image functors $R^i\pi_{1*}$ that take
a sheaf on $X\times \PP^m$ (or a complex of sheaves on $X\times \PP^m$) to
a sheaf on $X$ (or a complex of sheaves on $X$).  Further, if we are
willing to work with the derived category, then there is a single functor
$R\pi_{1*}$ that combines all of these higher direct image functors: the
functor $R\pi_{1*}$  takes a sheaf $\cF$ on $X\times \PP^m$ (or a complex
$\FF$ of sheaves) and returns an object in the derived category of
$\cO_X$-modules.  The functor $R\pi_{1*}$ combines the higher direct image
functors in the sense that, if $\GG$ is any complex that represents
$R\pi_{1*}\FF$, then $\HH_i(\GG)\cong R^{-i}\pi_{1*}\FF$ for all $i$.
In the special case where $X=\Spec(A)$, we will view each $R^i\pi_{1*}\cF$
as an $A$-module (instead of writing $\Gamma(X, R^i\pi_{1*}\cF)$), and
similarly for $R\pi_{1*}$.
If $\cF$ is an $\cO_{X\times \PP^m}$-module, then $R^i\pi_{1*}\cF=0$ for
all $i < 0$. Since computing $R\pi_{1*}\cF$ depends only on the
quasi-isomorphism class
of $\cF$, the same fact holds for any (locally free) resolution $\FF$ of an $\cO_{X\times
\PP^m}$-module. 

Let $\pi_2$ be the second projection $X\times \PP^m\to \PP^m$.  Given a
sheaf $\cG$ on $X$ and a sheaf $\cL$ on $\PP^m$, we set
\[
\cG \boxtimes \cL := \pi_1^*\cG\boxtimes \pi_2^*\cL.
\]
If $\cL=\cO_{\PP^m}(-e)$ is a line bundle on $\PP^m$, then by way of the projection formula~\cite[III, Ex.~8.3]{hartshorne}, computing $R\pi_{1*}(\cG\boxtimes \cL)$ is straightforward, and we will use this computation repeatedly.   There are three cases, depending on the value of $e$. 
\begin{enumerate}
    \item If $-e\geq 0$, then the only nonzero cohomology of
    $\cO_{\PP^m}(-e)$ is $H^0({\PP^m}, \cO_{\PP^m}(-e))$, and we have that
     $R\pi_{1*}((\cG\boxtimes \cO_{\PP^m}(-e))$ is the complex consisting of
    the sheaf $\cG\otimes H^0(\PP^m, \cO_{\PP^m}(-e))$ in homological degree $0$.
    \item If $-1\geq -e \geq -m$, then $\cO_{\PP^m}(-e)$ has no cohomology,
    so $R\pi_{1*}((\cG\boxtimes \cO_{\PP^m}(-e))=0$. 
    \item  If $-m-1\geq -e$, then the only nonzero cohomology of
    $\cO_{\PP^m}(-e)$ is $H^m({\PP^m}, \cO_{\PP^m}(-e))$, and we have that
    $R\pi_{1*}((\cG\boxtimes
    \cO_{\PP^m}(-e))$ is the complex consisting of sheaf $\cG\otimes
    H^m(\PP^m, \cO_{\PP^m}(-e))$ in homological degree $-m$.
\end{enumerate}

\section{Extremal Betti Tables in the Graded case}
\label{sec:graded}
In this section, we first prove Claim~\ref{cond:1}, providing a sufficient
condition for extremality in the graded case.  We then focus on the Eisenbud--Schreyer
construction of pure resolutions.

We assume throughout this section $\kk$ is an infinite field.
By~\cite{eis-erm}*{Lemma~9.6}, this assumption will not affect questions
related to cones of Betti tables.
A strictly increasing sequence of integers $d = (d_0<d_1<\cdots
<d_n)\in\ZZ^{n+1}$ is called a \defi{degree sequence} of $S$. We say a free
resolution $\FF$ is \defi{pure of type $d$} if it has the form 
\[
\FF:\quad 
    S(-d_0)^{\beta_0}\from S(-d_1)^{\beta_1}\from \cdots \from
    S(-d_n)^{\beta_n}\from 0.\qedhere
\] 

\begin{proof}[Proof of Claim~\ref{cond:1}]
Our argument follows~\cite{boij-sod1}*{\S2.1}, which extends a computation
of Herzog and K\"uhl~\cite{herzog-kuhl}; see
also~\cite{eis-schrey-icm}*{Proposition~2.1}. 

Let $M$ be a finite length module with a pure resolution
\[
0\gets M\gets S(-d_0)^{\beta_{0,d_0}}\gets S(-d_1)^{\beta_{1,d_1}}\gets
\dots \gets S(-d_n)^{\beta_{n,d_n}}\gets 0.
\]
Suppose that
$\beta(M)=\beta(M')+\beta(M'')$. Since $M$ has finite length, it follows
that $M'$ would also have to be a finite length module (the Hilbert series
is determined by the Betti table, and is additive).  Thus, by the
Auslander--Buchsbaum Theorem, the projective dimension of $M'$ is
 $n$. It then
follows from the decomposition of $\beta(M)$ that $M'$ admits a pure
resolution of type
$(d_0<d_1<\dots<d_n)$.  Thus, if the Betti table of a pure resolution is
unique up to scalar multiple, then $\beta(M')$ will be a scalar multiple of
$\beta(M)$.

To prove that the $\beta_{i,d_i}$ are determined (up to scalar multiple),
we consider the Herzog--K\"uhl equations for $M$ from \cite{herzog-kuhl}.
Since $M$ has finite length, the following $n$ equations must vanish:

\begin{align}\label{eq:HK}
\begin{cases}
\sum_{i=0}^n (-1)^i\beta_{i,d_i}&=0;\\
\sum_{i=0}^n (-1)^id_i\beta_{i,d_i}&=0;\\
\vdots &\vdots\\
\sum_{i=0}^n (-1)^id_i^{n-1}\beta_{i,d_i}&=0.
\end{cases}
\end{align}
Thinking of this as a system of $n$ linear equations in the
$(n+1)$-unknowns $\beta_{i,d_i}$, the solutions are given by the
kernel of the matrix
\[
\begin{pmatrix}
1&-1&\dots&(-1)^n\\
d_0&-d_1&\dots&(-1)^nd_n\\
\vdots & &\ddots &\vdots\\
d_0^{n-1}&-d_1^{n-1}&\dots&(-1)^nd_n^{n-1}
\end{pmatrix}.
\]
This is a rank $n$ matrix; in fact, the $n\times n$ minor given by the
first $n$ columns is nonzero. To see this, rescale every other column by
$-1$ to obtain an $n\times n$ Vandermonde matrix for $(d_0,\dots,d_{n-1})$.
Since the $d_i$ are strictly increasing, this Vandermonde determinant is
nonzero.  It thus follows that the kernel of this matrix has rank $1$, so
the $\beta_{i,d_i}$ are uniquely determined, up to scalar multiple.
\end{proof}

\begin{remark}\label{rmk:HK}
Using Cramer's rule and the formula for Vandermonde determinants,
any solution $(\beta_{0,d_0},\beta_{1,d_1},\dots, \beta_{n,d_n})$ 
to the system~\eqref{eq:HK} is a scalar multiple of
\[
\left( \frac{1}{\prod_{j\neq 0} |d_0-d_j|}, \frac{1}{\prod_{j\neq 1}
|d_1-d_j|}, \dots, \frac{1}{\prod_{j\neq n} |d_n-d_j|}\right).
\]
\end{remark}

We now show that any degree sequence of $S$ is realized by a pure
resolution.  The first two constructions of pure resolutions are due to
Eisenbud, Fl{\o}ystad, and Weyman~\cite{efw}.  Their constructions are based
on representation theory and Schur functors, and they thus require that
$\Bbbk$ has characteristic $0$.  See \cite{floy-expository}*{\S3} for an
expository treatment of those constructions.  The first characteristic-free
construction is due to Eisenbud and Schreyer~\cite{eis-schrey1}.  Their
construction, which relies on a spectral sequence or, equivalently, on the
Kempf-Lascoux-Weyman Geometric Technique, was later generalized in
\cite{beks-tensor}.

\begin{thm}[\cite{eis-schrey1}*{Theorem~5.1}]
\label{thm:exist std}
For any degree sequence $d=(d_0 < d_1 < \cdots < d_n)$, there exists a
finite length graded $S$-module whose minimal free resolution is pure of
type $d$. 
\end{thm}

Of course, it suffices to prove the theorem in the case where $d_0=0$, as
we can obtain a pure resolution of type $(d_0<\cdots<d_n)$ by tensoring a
pure resolution of type $(0<d_1-d_0<\dots<d_n-d_0)$ with $S(-d_0)$.  When
Boij and S\"oderberg conjectured the existence of pure resolutions, there
were very few known examples. One family of examples that was known came
from the Eagon--Northcott complex, the Buchsbaum--Rim complex, and other
related complexes~\cite{buchs-eis}.  Lascoux had shown that these complexes could be
constructed by applying a pushforward construction to a Koszul
complex~\cite{lascoux}. This pushforward construction has the effect of
collapsing strands of the Koszul complex, and Eisenbud and Schreyer
realized that (with the appropriate setup) this collapsing effect could be
iterated. This became the key to their construction of pure
resolutions.\footnote{The idea that Eisenbud and Schreyer's construction of
pure resolutions is a higher-dimensional analogue of the Eagon--Northcott
and Buchsbaum--Rim complexes is developed explicitly
in~\cite{beks-tensor}*{\S10}.}

Before presenting Eisenbud and Schreyer's general construction for a pure
resolution, we review the original collapsing technique in the following lemma.
This produces a pure resolution of type $(0,q+1,\dots, q+n)$,
which is the Eagon--Northcott complex
for an $n\times (q+1)$ matrix of linear forms over $\kk[x_1, \dots, x_{n+q}]$.
The proof of this lemma contains all of the technical features required for the
general case.  An example is provided in Figure~\ref{fig:0345twists}.

\begin{figure}
\begin{tabular}{|c||c|c|}%
\multicolumn{3}{c}{The complex $\KK$}\\%
\hline
&$\Spec(S)$&$\times\, \PP^3$ \qquad  \; \\ \hline
$\KK_0$&$S^1$\phantom{$(-0)$}&$\boxtimes\, \cO_{\PP^3}$\phantom{$(-0)$} \\%
$\KK_1$&$S(-1)^6$ &$\boxtimes\,  \cO_{\PP^3}(-1)$ \\%
$\KK_2$&$S(-2)^{15}$ &$\boxtimes\,  \cO_{\PP^3}(-2)$ \\%
$\KK_3$&$S(-3)^{20}$ &$\boxtimes\,  \cO_{\PP^3}(-3)$ \\%
$\KK_4$&$S(-4)^{15}$ &$\boxtimes\,  \cO_{\PP^3}(-4)$ \\%
$\KK_5$&$S(-5)^6$ &$\boxtimes\,  \cO_{\PP^3}(-5) $\\%
$\KK_6$&$S(-6)^1$ &$\boxtimes\,  \cO_{\PP^3}(-6)$ \\%
\hline%
\end{tabular}
$\xrightarrow{\quad R\pi_{2*} \quad }$
\begin{tabular}{|c||c|}%
\multicolumn{2}{c}{The complex $\FF$}\\%
\hline
&$\Spec(S)$ \\ \hline
$\FF_0$&$S^1\otimes H^0(\PP^3,\cO_{\PP^3})$ \\%
-&- \\%
-&- \\%
-&- \\%
$\FF_1$&$S(-4)^{15}\otimes H^3(\PP^3, \cO_{\PP^3}(-4))$ \\%
$\FF_2$&$S(-5)^{6\phantom{1}}\otimes H^3(\PP^3, \cO_{\PP^3}(-5)) $\\%
$\FF_3$&$S(-6)^{1\phantom{1}}\otimes H^3(\PP^3, \cO_{\PP^3}(-6))$ \\%
\hline%
\end{tabular}
\caption{To construct a pure resolution $\FF$ of type $(0,4,5,6)$ on
$\Spec(S')$, we begin with a Koszul complex $\KK$ on
$\Spec(S')\times \PP^3$ and then use a pushforward construction to collapse
three of the terms.  A term $\KK_i$ gets collapsed if the second factor is
a line bundle on $\PP^3$ with no cohomology.}
\label{fig:0345twists}
\end{figure}

\begin{lemma}\label{lem:0q}
Let $q$ be a positive integer and let $S':=\kk[x_1, \dots, x_{n+q}]$. Let
$f_1, \dots, f_{n+q}$ be generic bilinear forms on $\Spec(S')\times \PP^q$ 
and let $\KK$ be the Koszul complex of locally free sheaves on $\Spec(S')\times \PP^q$ given by the $f_i$. 
Then $R\pi_{1*}(\KK)$ is represented by a pure
resolution $\FF$ of type $(0,q+1,q+2,\dots,q+n)$ that resolves a
Cohen--Macaulay $S'$-module of codimension $n$.
\end{lemma}

\begin{proof}[Proof of Lemma~\ref{lem:0q}]
Since $\kk$ is infinite, we may assume that the $f_i$ form a regular
sequence, and hence they define a $q$-dimensional
subscheme $Z \subseteq
\mathbb A^{n+q}\times \PP^q$.  The Koszul complex
$\KK$ is thus a resolution of $\cO_Z$.  The support of ${\pi_1}_*\cO_Z$ has
dimension at most $q$, and therefore has codimension at least
$n$.  In fact, we will later see that the $S'$-module ${\pi_1}_*\cO_Z$ is a
Cohen--Macaulay of codimension $n$.

For $0\leq i\leq n+q$, the $\PP^q$-degree of the generators of $\KK_i$ is $i$.
By taking the direct images under the map $\pi_1: \Spec(S')\times \PP^q\to
\Spec(S')$,
we will collapse the terms $\KK_1, \KK_2, \dots, \KK_q$, resulting in the
desired pure resolution.

Our first goal is to show that $R^\ell{\pi_1}_{*}\KK\ne 0$ if and only if
$\ell=0$.
We do this in two steps.  As noted in
Section~\ref{sec:background}, since $\KK$ is a resolution of $\cO_Z$, it
follows that $R^\ell{\pi_1}_{*}\KK\ne 0$ only if $\ell\geq 0$. 

By computing $R{\pi_1}_{*}(\KK)$ in a second way, we will now show that
$R^\ell{\pi_1}_{*}\KK \ne 0$ only if $\ell\leq 0$.  Note that
$\KK_i={S'}^{\binom{n+q-1}{i}}(-i)\boxtimes \cO_{\PP^{q}}(-i)$.  For each
$i$, let $C^{-i,\bullet}$ be the \v{C}ech resolution of $\KK_i$ with
respect to the standard \v{C}ech cover $\{\Spec({S'})\times U_0, \dots
\Spec({S'})\times U_{q}\}$ of $\Spec({S'})\times \PP^{q}$.  Since the
construction of \v{C}ech resolutions is functorial, we obtain a double
complex $C^{\bullet,\bullet}$ consisting of $\pi_{1*}$-acyclic sheaves on
$\Spec({S'})\times \PP^q$, which has the form:
\[
\xymatrix@C=1.25em@R=1.25em{
C^{\bullet,\bullet}:
& 
\vdots 
&& 
\vdots
&
\\
0 
& 
{S'}\boxtimes\left(\bigoplus_{k,k'=0}^{q} \cO|_{U_k\cap U_{k'}}\right)
\ar[l] \ar[u] 
&&
{S'}(-1)^{n+q}\boxtimes \left(\bigoplus_{k,k'=0}^{q} \cO(-1)|_{U_k\cap
U_{k'}}\right) \ar[ll]
\ar[u]
&
\dots \ar[l] \\
0 
& 
{S'}\boxtimes \left(\bigoplus_{k=0}^{q} \cO|_{U_k}\right) \ar[l] \ar[u] 
&&
{S'}(-1)^{n+q}\boxtimes \left(\bigoplus_{k=0}^{q} \cO(-1)|_{U_k}\right)
\ar[ll]
\ar[u]
&
\dots \ar[l] \\ 
& 
0\ar[u]
&&
0\ar[u]
&
\\
}
\]
We may now compute $R\pi_{1*}\KK$ by applying
$\pi_{1*}$ to this double complex $C^{\bullet,\bullet}$ and running the
vertical 
spectral sequence for the resulting double complex of ${S'}$-modules.
After taking vertical homology of $C^{\bullet,\bullet}$, we obtain the
${}_vE^{\bullet,\bullet}_1$-page with differential
$\partial_1^{\bullet,\bullet}$.
\[
{}_vE^{\bullet,\bullet}_1:
\xymatrix@C=1.25em@R=.8em{
& 
\vdots 
&& 
\vdots
&
\\
0 
& 
{S'}\otimes H^1(\PP^{q},\cO) \ar[l]
&&
{S'}(-1)^{n+q}\otimes H^1(\PP^{q},\cO(-1)) \ar[ll]_-{\partial_1^{-1,1}}
&
\dots \ar[l] \\
0 
& 
{S'}\otimes H^0(\PP^{q},\cO) \ar[l]  
&&
{S'}(-1)^{n+q}\otimes H^0(\PP^{q},\cO(-1)) \ar[ll]_-{\partial_1^{-1,0}}
&
\dots \ar[l] \\ 
& 
0
&&
0
&
\\
}
\]
The general entry on the ${}_vE_1$-page is given by 
\[
{}_vE_1^{-i,j} = {S'}(-i)^{\binom{n+q}{i}}\otimes H^j(\PP^q,\cO(-i)).
\]
Since $H^j(\PP^q,\cO(-i))=0$ unless $j=0$ or $q$, most of these entries of
${}_vE_1$ are equal to $0$.  In fact, ${}_vE_1$ has a single nonzero entry
on row $0$, with the only remaining nonzero entries appearing on row $q$,
as shown below. 
\[
{}_vE_1^{\bullet,\bullet}:\xymatrix@C=1.25em@R=.8em{
0&0\ar[l]&0\ar[l]&\dots\ar[l]&{S'}(-q-1)^{\binom{n+q}{q+1}}\otimes
H^q(\cO(-q-1))\ar[l]^{}&\dots\ar[l]_-{\ \ \ \partial_1^{-q-2,q}}\\
0&0\ar[l]&0\ar[l]&\dots\ar[l]&0\ar[l]&\dots\ar[l]\\
&\vdots &\vdots & &\vdots &\\
0&0\ar[l]&0\ar[l]&\dots\ar[l]&0\ar[l]&\dots\ar[l]\\
0&{S'}\otimes H^0(\cO)\ar[l]&0\ar[l]&\dots\ar[l]&0\ar[l]&\dots\ar[l]
}
\]
Since all of the terms of the ${}_vE_1$ page lie in total degree $-i+j\leq
0$, we see that $R^\ell\pi_{1*}\KK\ne 0$ only if $\ell\leq 0$, as claimed.  

Note that after passing the ${}_vE_1$-page, the only other differential
exiting or entering a nonzero term will occur on ${}_vE_{q+1}$, from
position $(-i,j)=(-q-1,q)$ to $(-i,j)=(0,0)$.  Since this spectral sequence
satisfies ${}_vE_1^{-i,j} \Rightarrow R^{-i+j}\pi_{1*}\KK$ and our previous
computation shows that $R^\ell\pi_{1*}\KK\ne 0$ if only if $\ell= 0$, only
positions $(0,0)$ and $(-q-1,q)$ may contain nonzero entries on the
${}_vE_2$-page. In addition, since $R^{-1}\pi_{1*}\KK=0$, we see that
${}_vE^{-q-1,q}_\infty=0$. Hence the differential $\partial_{q+1}^{-q-1,q}$
must be injective. 
\[
{}_vE_{q+1}^{\bullet,\bullet}:
\xymatrix@C=1.25em@R=.8em{
0&0&0&\dots&\coker
\partial_1^{-q-2,q}\ar@{^{(}->}[llldd]^-{\partial_{q+1}^{-q-1,q}}&0\\
&\vdots &\vdots & &\vdots &\\
0&{S'}\otimes H^0(\cO)&0&\dots&0&\dots
}
\]
The differential $\partial_{q+1}^{-q-1,q}$ lifts to a map $\phi$ of the
free modules on the ${}_vE_1$ page:
\[
\xymatrix{
S'\otimes H^0(\cO) &&&{S'}(-q-1)^{\binom{n+q}{q+1}}\otimes
H^q(\cO(-q-1))\ar@{-->}[lll]_-\phi\ar@{->>}[d]\\
S'\otimes H^0(\cO)\ar[u]_-{=}&&&\coker
\partial_1^{-q-2,q}\ar[lll]_-{\partial_{q+1}^{-q-1,q}}.
}
\]
We thus conclude that $R^0\pi_{1*}\KK = \pi_{1*}\cO_Z$ is represented by a
minimal complex of the form 
\[
\xymatrix@C=1.75em{
{S'} & \ar[l]_-{\phi} {S'}(-q-1)^{\binom{n+q}{q+1}}\otimes H^q(\cO(-q-1)) &
\ar[l] {S'}(-q-2)^{\binom{n+q}{q+2}}\otimes H^q(\cO(-q-2)) & \ar[l] \cdots
}. 
\]
Notice this is a pure complex of type $(0,q+1,\dots,q+n)$.  Since it
is acyclic, it is actually a
resolution of the ${S'}$-module $\pi_{1*}\cO_Z$. 
Hence this module has projective dimension 
$n$, and since we noted initially that it has codimension
at least $n$, it follows that $\pi_{1*}\cO_Z$ is a Cohen--Macaulay module
of codimension $n$. 
\end{proof}

The following proposition, due to Eisenbud and Schreyer, provides a more
general framework than Lemma~\ref{lem:0q} for collapsing terms from a
resolution.
The proof is nearly identical. 
See
Figure~\ref{fig:ESprop} for an illustration of this result. 

\begin{prop}[\cite{eis-schrey1}*{Proposition~5.3}]
\label{prop:pushforward}
Let $\cF$ be a sheaf on $X\times\PP^m$ that has a resolution $\mathbf{G}$
arising from $\cO_X$-modules $\cG_i$, such that  
\[
\mathbf{G}_i=\cG_i\boxtimes\cO(-e_i) \text{ for } 0\leq i\leq N 
\]
and $e_0<\cdots<e_N$. If this sequence contains the subsequence
$(e_{k+1},\dots,e_{k+m})=(1,2,\dots,m)$ for some $k\geq-1$, then 
\[
R^\ell \pi_{1*}\cF \cong R^\ell \pi_{1*}\mathbf{G}= 0 \quad\text{for
}\ell\ne 0,
\]
and $\pi_{1*}\cF$ has a resolution $\mathbf{G}'$, where
\[
\mathbf{G}'_i=\begin{cases}
\cG_i\otimes H^0(\PP^m, \cO(-e_i))& \text{ for } 0\leq i\leq k,\\
\cG_{i+m}\otimes H^m(\PP^m, \cO(-e_{i+m}))& \text{ for } k+1 \leq i\leq
N-m.  
\end{cases}
\]
\end{prop}

\begin{figure}
\begin{tabular}{|c||c|l|}%
\multicolumn{3}{c}{The complex $\mathbf{G}$}\\%
\hline
&$X$&$\times \,  \PP^2$\\ \hline
$\mathbf{G}_0$&$\mathcal G_0$&$\boxtimes\,  \cO_{\PP^2}(-e_0)$ \\%
$\mathbf{G}_1$&$\mathcal G_1$&$\boxtimes\,  \cO_{\PP^2}(-e_1)$ \\%
$\mathbf{G}_2$&$\mathcal G_2$&$\boxtimes\,  \cO_{\PP^2}(-e_2)$ \\
$\mathbf{G}_{3}$&$\mathcal G_{3}$&$\boxtimes \, \cO_{\PP^2}(-1)$ \\%
$\mathbf{G}_{4}$&$\mathcal G_{4}$&$\boxtimes \, \cO_{\PP^2}(-2)$ \\%
$\mathbf{G}_5$&$\mathcal G_5$&$\boxtimes\,  \cO_{\PP^2}(-e_5)$ \\
$\mathbf{G}_6$&$\mathcal G_6$&$\boxtimes\,  \cO_{\PP^2}(-e_6)$ \\
$\mathbf{G}_7$&$\mathcal G_7$&$\boxtimes\,  \cO_{\PP^2}(-e_7)$ \\
$\mathbf{G}_8$&$\mathcal G_8$&$\boxtimes\,  \cO_{\PP^2}(-e_8)$ \\
\hline%
\end{tabular}
$\xrightarrow{\quad Rp_{*} \quad }$
\begin{tabular}{|c||c|}%
\multicolumn{2}{c}{The complex $\mathbf{G}'$}\\%
\hline
&$X$\\ \hline
$\mathbf{G}'_0$&$\mathcal G_0\otimes H^0(\cO_{\PP^2}(-e_0))$ \\%
$\mathbf{G}'_1$&$\mathcal G_1\otimes H^0(\cO_{\PP^2}(-e_1))$ \\%
$\mathbf{G}'_2$&$\mathcal G_2\otimes H^0(\cO_{\PP^2}(-e_2))$ \\%
 - & -  \\%
- & - \\%
$\mathbf{G}'_3$&$\mathcal G_3\otimes H^2(\cO_{\PP^2}(-e_5))$ \\%
$\mathbf{G}'_4$&$\mathcal G_4\otimes H^2(\cO_{\PP^2}(-e_6))$ \\%
$\mathbf{G}'_5$&$\mathcal G_5\otimes H^2(\cO_{\PP^2}(-e_7))$ \\%
$\mathbf{G}'_6$&$\mathcal G_6\otimes H^2(\cO_{\PP^2}(-e_8))$ \\%
\hline%
\end{tabular}
\caption{Proposition~\ref{prop:pushforward} uses a pushforward and the 
vanishing cohomology of line bundles on $\PP^m$ to collapse terms from a
free resolution. The above illustrates the proposition
when $m=k=2$ and $N=8$.}
\label{fig:ESprop}
\end{figure}

\begin{proof}
We proceed in a matter similar to the proof of Lemma~\ref{lem:0q}. 
Our first goal is to show in two steps that $R^\ell p_{*}\mathbf{G}\ne 0$
if and only if $\ell=0$. First, since $\mathbf{G}$ is a resolution of $\cF$,
it follows that $R^\ell p_{*}\KK\ne 0$ only if $\ell\geq 0$. 

We now compute $R\pi_{1*}(\GG)$ in a second way to show that $R^\ell
\pi_{1*}\GG \ne 0$ only if $\ell\leq 0$.  For each $i$, let
$C^{-i,\bullet}$ be the \v{C}ech resolution of $\GG_i$ with respect to the
standard \v{C}ech cover $\{X\times U_0, \dots, X\times U_{m}\}$ of
$X\times \PP^{m}$.  Since the construction of \v{C}ech resolutions is
functorial, we obtain a double complex $C^{\bullet,\bullet}$ consisting of
$\pi_{1*}$-acyclic sheaves on $X\times \PP^m$.  To compute $R\pi_{1*}\GG$,
we apply $\pi_{1*}$ to the double complex $C^{\bullet,\bullet}$ and run the
vertical spectral sequence for the resulting double complex of
$\cO_X$-modules.
This yields an ${}_vE_1$-page with general entry 
\[
{}_vE_1^{-i,j} = \mathcal G_i\otimes H^j(\PP^m,\cO(-e_i)).
\]
Since $H^j(\PP^m,\cO(-e_i))=0$ unless $j=0$ or $m$, most of these entries
are equal to $0$.  In fact, the resulting ${}_vE_1$-page consists of a
strand of nonzero entries in row $0$, followed by all zeroes in columns
$k+1, \dots, k+m$, followed by a strand of nonzero entries in row $m$.
\[
{}_vE_1^{\bullet,\bullet}:\hspace{-1.25cm}
\xymatrix@C=1em@R=.8em{
& -i=\underline{k}& \underline{k+1}& &\underline{k+m}&\underline{k+m+1}&\\
\dots&0\ar[l]&0\ar[l]&\cdots\ar[l]&0\ar[l]&\mathcal G_{k+m+1}\otimes
H^m(\cO(-e_{k+m+1}))\ar[l]^{}&\dots\ar[l]_-{\qquad\quad
\partial_1^{-k-m-2,m}}\\
\dots&0\ar[l]&0\ar[l]&\dots\ar[l]&0\ar[l]&0\ar[l]&\dots\ar[l]\\
&\vdots&\vdots & &\vdots &\vdots &\\
\dots&0\ar[l]&0\ar[l]&\dots\ar[l]&0\ar[l]&0\ar[l]&\dots\ar[l]\\
\dots&\mathcal G_k\otimes
H^0(\cO(-e_k))\ar[l]&0\ar[l]&\dots\ar[l]&0\ar[l]&0\ar[l]&\dots\ar[l]
}
\]
Since all of the nonzero terms of this ${}_vE_1$-page lie in total
cohomological degree $-i+j\leq 0$, we see that $R^\ell \pi_{1*}\GG\ne 0$
only if $\ell\leq 0$, as desired.  

We have now nearly constructed our complex $\mathbf{G}'$.  The nonzero
entries on the ${}_vE_1$-page are precisely the terms we use in
$\mathbf{G}'$, and as its differential, we will use $\partial_1$ everywhere
except for the map $\mathbf{G}'_{k}\from \mathbf{G}'_{k+1}$:
\[
\xymatrix@C=2em{
\GG'_0&\GG'_1\ar[l]_-{\partial_1^{-1,0}}&\dots\ar[l]&\GG'_k\ar[l]_-{\partial_1^{-k,0}}&\GG'_{k+1}\ar[l]_-{\, ???}&&\GG'_{k+2}\ar[ll]_-{\partial_1^{-k-m-2,m}}&\dots\ar[l]
&\GG'_{N-m}\ar[l]_-{\partial_1^{-N,m}}&0.\ar[l]}
\]

To complete the construction of $\mathbf{G}'$ and to check its exactness,
we note that after the ${}_vE_2$-page, the only other differential exiting
or entering a nonzero term will occur on the ${}_vE_{m+1}$-page, from
position $(-i,j)=(-k-m-1,m)$ to $(-i,j)=(-k,0)$.  Since we have
${}_vE_1^{-i,j} \Rightarrow R^{-i+j}\pi_{1*}\GG$ 
and our previous computation shows that $R^\ell \pi_{1*}\GG\ne 0$ if and only if
$\ell= 0$, only positions $(-k,0)$ and $(-k-m-1,m)$ may contain nonzero
entries on the ${}_vE_2$-page. 
In particular, although we have not yet fully constructed the differential
for $\mathbf{G}'$, we already see that our complex is exact in every
position except possibly at $\GG'_k$ or $\GG'_{k+1}$.

We now examine the differential $\partial_{m+1}^{-k-m-1,m}$ on
${}_vE_{m+1}$. This differential must be an isomorphism when $k>0$, as
otherwise $R^{k+1}\pi_{1*}\GG$ and $R^{k}\pi_{1*}\GG$ would be nonzero.
When $k=0$, it must be injective for the same reason.
\[
{}_vE_{m+1}^{\bullet,\bullet}:\xymatrix@C=1.25em@R=1em{
0&0&0&\dots&\coker
\partial_1^{-k-m-2,m}\ar@{^{(}->}[llldd]^-{\partial_{m+1}^{-k-m-1,m}}&0\\
&\vdots &\vdots & &\vdots &\\
0&\ker \partial_1^{-k,1}&0&\dots&0&\dots
}
\]
This differential $\partial_{m+1}^{-k-m-1,m}$ lifts to a map
$\phi\colon \GG'_{k+1}\to \GG'_k$,
\[
\xymatrix{
\GG'_k=\cG_k\otimes H^0(\cO(-e_k)) &&&\GG_{k+1}'=\cG_k\otimes
H^m(\cO(-e_m))\ar@{-->}[lll]_-\phi\ar@{->>}[d]\\
\ker \partial_1^{-k,1}\ar@{_(->}[u]&&&\coker
\partial_1^{-k-m-2,m}\ar[lll]_-{\partial_{m+1}^{-k-m-1,m}}
}
\]
completing our construction of $\GG'$:
\[
\xymatrix{\GG'_0&\GG'_1\ar[l]_-{\partial_1^{-1,0}}&\dots\ar[l]&\GG'_k\ar[l]_-{\partial_1^{-k,0}}&\GG'_{k+1}\ar[l]_-{\phi}&&\GG'_{k+2}\ar[ll]_-{\partial_1^{-k-m-2,m}}&\dots\ar[l]
&\GG'_{N-m}\ar[l]_-{\partial_1^{-N,m}}&0.\ar[l]}
\]
It follows that $\GG'$ is exact at $\GG'_k$ and at $\GG'_{k+1}$.
Since $\GG'$ is acyclic, it follows that it is a resolution $\pi_{1*}\cF$, as
desired.
\end{proof}

Proposition~\ref{prop:pushforward} provides a tool to construct a pure free
resolution with a prescribed degree sequence. We illustrate this by
explaining how to construct a pure resolution of type $d=(0,3,5,6)$ over
$S=\Bbbk[x_1,x_2,x_3]$: see Figure~\ref{fig:bigone}.
Since the highest degree term has degree $6$, we define the ring
$S':=\kk[y_1, \dots, y_6]$ and consider a Koszul complex involving $6$
multilinear forms. The gaps in the degree sequence $d$ tell us how to
choose the projective spaces we use to collapse the various terms.  For
instance, this degree sequence has two gaps:  the gap between $0$ and $3$
consisting of the integers $\{1,2\}$; and the gaps between $3$ and $5$
consisting of $\{4\}$.  To collapse degrees $1$ and $2$, we will
use a copy of $\PP^2$; to collapse degree $4$, we will use a copy of
$\PP^1$.

We thus define a Koszul complex $\KK$ involving $6$ multidegree
$(1,1,1)$-forms on $\Spec(S')\times \PP^2\times \PP^1$, and we set
$\GG:=\KK\otimes_{\cO_{\Spec(S')\times \PP^2\times \PP^1}} (\cO_{\Spec S'}\boxtimes
\cO_{\PP^2}\boxtimes\cO_{\PP^1}(3))$. This twist of the Koszul complex is
engineered so that we are able collapse the proper terms, as shown in
Figure~\ref{fig:bigone}. Put another way, we have attached a line bundle
with vanishing cohomology to each of the terms in $\GG$ that we want to
collapse. By applying Proposition~\ref{prop:pushforward} twice to $\GG$, we
obtain a pure resolution of type $(0,3,5,6)$ on $\Spec S'$ that resolves a
Cohen--Macaulay module of codimension $3$.  Finally, we mod out by $3$
generic linear forms to obtain a pure resolution $\FF$ of type $(0,3,5,6)$
on $\Spec(S)$ that resolves a module of finite length: 
\[
\FF =  \bigg[ S^4\gets S(-3)^{20}\gets S(-5)^{36}\gets S(-6)^{20}\gets 0
\bigg]. 
\] 

\begin{figure}
\begin{tabular}{|c||c|l|l|}%
\multicolumn{4}{c}{The original complex $\mathbf{G}$}\\%
 \hline
&$\Spec(S')$&$\times \, \PP^2$&$\times \, \PP^1$\\ \hline
$\mathbf{G}_0$&$(S')^1$&$\boxtimes\,  \cO_{\PP^2}$&$\boxtimes\,
\cO_{\PP^1}(3)$ \\%
$\mathbf{G}_1$&$S'(-1)^6$ &$\boxtimes\,  \cO_{\PP^2}(-1)$&$\boxtimes\,
\cO_{\PP^1}(2)$ \\%
$\mathbf{G}_2$&$S'(-2)^{15}$ &$\boxtimes\,  \cO_{\PP^3}(-2)$&$\boxtimes\,
\cO_{\PP^1}(1)$ \\%
$\mathbf{G}_3$&$S'(-3)^{20}$ &$\boxtimes\,  \cO_{\PP^3}(-3)$&$\boxtimes\,
\cO_{\PP^1}$\\%
$\mathbf{G}_4$&$S'(-4)^{15}$ &$\boxtimes \, \cO_{\PP^3}(-4)$ &$\boxtimes\,
\cO_{\PP^1}(-1)$ \\%
$\mathbf{G}_5$&$S'(-5)^6$ &$\boxtimes\,  \cO_{\PP^3}(-5) $&$\boxtimes\,
\cO_{\PP^1}(-2)$ \\%
$\mathbf{G}_6$&$S'(-6)^1$ &$\boxtimes \, \cO_{\PP^3}(-6)$ &$\boxtimes\,
\cO_{\PP^1}(-3)$ \\%
\hline%
\end{tabular}

$\xymatrix{ \ar[d]_{R\pi_{3*}} \\ \  }$

\begin{tabular}{|c||c|l|}%
\multicolumn{3}{c}{The complex $\mathbf{G}'$ after one projection}
\\ \hline
&$\Spec(S')$&$\times\,  \PP^1$\\ \hline
$\mathbf{G}'_0$&$(S')^1$&$\boxtimes \, \cO_{\PP^2} \otimes
H^0(\cO_{\PP^1}(3))$ \\%
$\mathbf{G}'_1$&$S'(-1)^6$ &$\boxtimes\,  \cO_{\PP^2}(-1) \otimes
H^0(\cO_{\PP^1}(2))$ \\%
$\mathbf{G}'_2$&$S'(-2)^{15}$ &$\boxtimes\,  \cO_{\PP^3}(-2)\otimes
H^0(\cO_{\PP^1}(1))$ \\%
$\mathbf{G}'_3$&$S'(-3)^{20}$ &$\boxtimes\,  \cO_{\PP^3}(-3)\otimes H^0(
\cO_{\PP^1})$\\%
-& -& \hspace{1.8cm} - \\%
$\mathbf{G}'_4$&$S'(-5)^6$ &$\boxtimes\,  \cO_{\PP^3}(-5) \otimes
H^1(\cO_{\PP^1}(-2))$ \\%
$\mathbf{G}'_5$&$S'(-6)^1$ &$\boxtimes\,  \cO_{\PP^3}(-6)\otimes
H^1(\cO_{\PP^1}(-3))$ \\%
\hline%
\end{tabular}

$\xymatrix{ \ar[d]_{R\pi_{2*}} \\ \  }$

\begin{tabular}{|c||c|}%
\multicolumn{2}{c}{The pure resolution $\FF$}
\\ \hline
&$\Spec(S')$\\ \hline
$\FF_0$&$(S')^1\otimes H^0(\cO_{\PP^2}) \otimes H^0(\cO_{\PP^1}(3))$ \\%
-&- \\%
-&- \\%
$\FF_1$&$S'(-3)^{20}\otimes H^2(\cO_{\PP^3}(-3))\otimes H^0(
\cO_{\PP^1})$\\%
-& - \\%
$\FF_2$&$S'(-5)^6\otimes H^2(\cO_{\PP^3}(-5)) \otimes H^1(\cO_{\PP^1}(-2))$
\\%
$\FF_3$&$S'(-6)^1\otimes H^2(\cO_{\PP^3}(-6))\otimes H^1(\cO_{\PP^1}(-3))$
\\%
\hline%
\end{tabular}

\caption{
We iterate Proposition~\ref{prop:pushforward} to build a pure
resolution $\FF$ of type $(0,3,5,6)$ over $S'$.   Modding out by
linear forms yields a resolution over $S$.
}
\label{fig:bigone}
\end{figure}

\begin{proof}[Proof of Theorem~\ref{thm:exist std}]
Without loss of generality, we may assume that $d_0=0$. 
We define $S'=\kk[y_1, \dots, y_{d_n}]$.  It suffices to construct
a Cohen--Macaulay $S'$-module of codimension $n$ with a pure resolution of
type $d$, as we may then mod out by generic linear forms to obtain a pure
resolution of a finite length $S$-module.

We define an auxiliary space $\PP$ which is a
product of projective spaces corresponding to the gaps in the degree
sequence $d=(d_0<d_1<\dots<d_n)$. To record these gaps, set 
\[
m_i:= d_i-d_{i-1}-1 \quad\text{for } 1\leq i\leq n.
\]
Set $\PP:=\PP^{m_1}\times \dots \times \PP^{m_n}$, which has dimension
$d_n-n$. 
Choose $d_n$ generic multilinear forms of multidegree
$(1,1,\dots,1)$. Since $\Bbbk$ is an infinite field, these forms give a
regular sequence.\footnote{In fact, this is also true over a finite field
by~\cite{eis-schrey1}*{Proposition~5.2}.}  Let $\KK$ denote the Koszul
complex on these multilinear forms, and define
\[
\GG:=\KK\otimes (\cO_{\Spec(S')} \boxtimes \cO_{\PP^{m_1}}\boxtimes
\cO_{\PP^{m_2}}(-d_1)\boxtimes  \dots \boxtimes \cO_{\PP^{m_n}}(-d_{n-1})).
\]
Note that $\GG$ is an exact complex with 
\[
\GG_i = S'(-i)^{\binom{d_n}{i}}\boxtimes \cO_{\PP^{m_1}}(-i)\boxtimes
\cO_{\PP^{m_2}}(-d_1-i)\boxtimes \dots \boxtimes
\cO_{\PP^{m_n}}(-d_{n-1}-i). 
\]

By repeatedly applying Proposition~\ref{prop:pushforward} (the order in
which we pushforward does not matter), all terms from $\GG$ will eventually
be collapsed away with exception of $\GG_{d_i}$ for $0 \leq i\leq n$.  More
precisely, when we push away from $\PP^{m_i}$,
Proposition~\ref{prop:pushforward} implies that we will collapse away the
terms that originally corresponded to $\GG_{d_{i}+1}, \dots
\GG_{d_{i+1}-1}$.

This process produces a pure resolution $\FF$ of graded $S'$-modules, where
\[
\FF_k = S'(-d_k)^{\binom{d_n}{k}}\otimes \bigotimes_{i=1}^{k-1}
H^0(\PP^{m_i},\cO(-d_{i-1}-k)) \otimes
\bigotimes_{i=k}^{n}H^{m_i}(\PP^{m_i},\cO(-d_{i-1}-k)). 
\]
Since $\GG$ resolves a module of codimension $d_n$ and the fibers of the
projection $p: X\times \PP\to X$ have dimension $d_n-n$, it follows that
the cokernel of $\FF$ has support of codimension at least $n$. However,
since $\FF$ is a resolution of projective dimension $n$, we conclude that
the cokernel of $\FF$ is a Cohen--Macaulay $S'$-module of codimension $n$,
as desired. 
\end{proof}

If one works with the base scheme $\Proj(S)$ instead of $\Spec(S)$, then
there is a slightly different argument which eliminates the need to pass to
the intermediate ring $S'$, but this requires different steps to check
exactness.  This was Eisenbud and Schreyer's original approach
in~\cite{eis-schrey1}*{\S5}.

\begin{remark}
There is a useful shorthand for reverse-engineering the Eisenbud--Schreyer
construction of a pure resolution. For instance, to construct a pure
resolution of type $(0,3,5,6)$, begin by considering the table on the left,
where we have marked with an asterisk the degrees that we need to collapse.
We may then use a copy of $\PP^2$ to collapse the first two asterisks and a
copy of $\PP^1$ to collapse the last asterisk. To do so, line up integers
as in the middle table so that the vanishing cohomology degrees of $\PP^2$
and $\PP^1$ align with the asterisks.  Now fill in the remaining entries of
the table linearly.
\[
\begin{tabular}{|c|c|} \hline
$\Spec(S')$&\\ \hline
\phantom{-}0&\\ \hline
-1&$\ast$\\ \hline
-2&$\ast$\\ \hline
-3&\\ \hline
-4&$\ast$\\ \hline
-5&\\ \hline
-6&\\ \hline
\end{tabular}
\xrightarrow{\hspace{1.5cm}}
\begin{tabular}{|c|c|c|} \hline
$\Spec(S')$&$\PP^2$&$\PP^1$\\ \hline
\phantom{-}0&&\\ \hline
-1&-1&\\ \hline
-2&-2&\\ \hline
-3&&\\ \hline
-4&&-1\\ \hline
-5&&\\ \hline
-6&&\\ \hline
\end{tabular}
\xrightarrow{\hspace{1.5cm}}
\begin{tabular}{|c|c|c|} \hline
$\Spec(S')$&$\PP^2$&$\PP^1$\\ \hline
\phantom{-}0&\phantom{-}0&\phantom{-}3\\ \hline
-1&-1&\phantom{-}2\\ \hline
-2&-2&\phantom{-}1\\ \hline
-3&-3&\phantom{-}0\\ \hline
-4&-4&-1\\ \hline
-5&-5&-2\\ \hline
-6&-6&-3\\ \hline
\end{tabular}
\]
The last table tells us that we should build a Koszul complex of six
$(1,1,1)$-forms on $\Spec(S')\times \PP^2\times \PP^1$ and then twist by
the degrees we see in the top row: $\cO(0,0,3)$.  Note that this is precisely
the construction from Figure~\ref{fig:bigone}.
\end{remark}

Although the proof of Theorem~\ref{thm:exist std} is constructive, it does
not provide an efficient technique for understanding the differentials of
the resulting pure resolution $\FF$.  To obtain explicit formulas for the
differentials from the proof, we would have to carry a description of the
differential through the spectral sequence.

A more efficient approach to understanding the differentials of these
Eisenbud--Schreyer pure resolutions is given in~\cite{beks-tensor}*{\S4}.
That article constructs a generic version of the Eisenbud--Schreyer pure
resolution, referred to as a \defi{balanced tensor complex}, which is
defined over a polynomial ring in many more variables.  The differentials
for the tensor complex can be expressed in terms of explicit multilinear
constructions (e.g., (co)multiplication maps on symmetric and exterior
products, among others). Since the Eisenbud--Schreyer pure resolutions are
obtained as specializations of balanced tensor
complexes~\cite{beks-tensor}*{Theorem~10.2}, this construction provides
closed formulas for the various differentials in the Eisenbud--Schreyer
pure resolutions.

\begin{example}
There is a Macaulay2 package {\tt TensorComplexes} that can be used to
compute the Eisenbud--Schreyer pure resolutions explicitly~\cite{M2}.  With
$\kk=\mathbb F_{101}$, the following code computes the first differential
for a pure resolution of type $(0,1,3,5)$.
\begin{verbatim}
i1 : loadPackage "TensorComplexes";
i2 : FF = pureResES({0,1,3,5},ZZ/101);
i3 : betti FF
            0  1  2 3
o3 = total: 8 15 10 3
         0: 8 15  . .
         1: .  . 10 .
         2: .  .  . 3

i4 : FF.dd_1
o4 = | x_0 0   0   0   x_1 0   0    0   x_2 0   0   0   0   0   0   |
     | 0   x_0 0   0   0   x_1 0    0   0   x_2 0   0   0   0   0   |
     | 0   0   x_0 0   0   0   x_1  0   0   0   x_2 0   0   0   0   |
     | 0   0   0   x_0 0   0   0    x_1 0   0   0   x_2 0   0   0   |
     | 0   0   0   x_2 x_0 0   -x_2 0   x_1 x_2 0   0   0   0   0   |
     | 0   0   0   0   0   x_0 0    0   0   x_1 0   0   x_2 0   0   |
     | 0   0   0   0   0   0   x_0  0   0   0   x_1 0   0   x_2 0   |
     | 0   0   0   0   0   0   0    x_0 0   0   0   x_1 0   0   x_2 |
\end{verbatim}
\end{example}

\section{Extremal Betti Tables in the Local case}
\label{sec:local}

Let $M$ be a finitely generated module over a regular local ring $R$ of
dimension $n$. From the minimal free resolution of $M$, 
\[
0 \gets M \gets R^{\beta^R_0(M)}\gets R^{\beta^R_1(M)}\gets \cdots\gets
R^{\beta^R_n(M)}\gets 0,
\]
we obtain the \defi{(local) Betti table} $\beta^R(M)=(\beta^R_0(M), \dots,
\beta^R_n(M))$. Here we again restrict our attention to the case when $M$ is
of finite length. 

As in the graded case, we would like to find modules $M$ of finite length
where $\beta^R(M)$ is extremal.  However, unlike the graded case, there are
no natural candidates for such vectors. It turns out that this is because
no local Betti table is extremal.

\begin{sufficient condition}\label{cond:2}
If $\dim(R)>1$, then there does not exist any $R$-module $M$ of finite
length whose Betti table is extremal.
\end{sufficient condition}

\begin{example}
Let $R=\kk[[x,y]]$.  Then the local Betti table of the residue field $\kk$
is $\beta^R(\kk)=(1,2,1)$.  If $M=R/\langle x^2,xy,y^2\rangle $ and $N=\Hom(M,\kk)$, then
we have the decomposition
\[
(1,2,1)=\beta^R(\kk)=\tfrac{1}{3}\beta^R(M)+\tfrac{1}{3}\beta^R(N)=\tfrac{1}{2}(1,3,2)+\tfrac{1}{2}(2,3,1).
\]
\end{example}

To understand how this comes to pass, we now assume that $n>1$ and view
each $\beta^R(M)\in \QQ^{n+1}$. An extremal local Betti table corresponds
to a ray of the cone 
\[
B_\QQ^\text{fin}(R) := \QQ_{\geq 0}\cdot  \left\{ \beta^R(M) \mid M \text{
is an } R\text{-module of finite length}\right\} \subset \QQ^{n+1}.
\] 

\begin{thm}[\cite{beks-local}*{Theorem~1.1}]
\label{thm:local}
If $R$ is an $n$-dimensional regular local ring with $n>1$, then
$B_\QQ^\text{fin}(R)$ is an open cone that has no extremal rays. More
precisely, 
\[
B_\QQ^\text{fin}(R) = 
\QQ_{>0} \cdot \{ \rho_0,\rho_1,\dots,\rho_{n-1}\},
\]
where $\rho_i = e_i+e_{i+1}$ is the sum of the $i$th and $(i+1)$st standard
basis vectors of $\QQ^{n+1}$.
\end{thm}
The story for finitely generated modules is similar;
see~\cite{beks-local}*{\S4}. Of course, if $\dim(R)=1$, then $\rho_0$ is an extremal ray, as it spans
the entire cone. 

\begin{proof}[Proof of Theorem~\ref{thm:local}]
For brevity, set $C:= \QQ_{>0} \cdot \{
\rho_0,\rho_1,\dots,\rho_{n-1}\}$. 
Clearly the cone $C$ lies in the linear subspace of $\QQ^{n+1}$ defined by
$\sum_{k=0}^n (-1)^k\beta^R_k=0$.  Inside this subspace, an elementary
computation confirms that $C$ equals the open cone defined by the
inequalities:
\[
0<\sum_{k=i}^n(-1)^{i-k}\beta^R_k \quad \text{for } 1 \leq i\leq n.  
\]
When applied to the Betti numbers of a module $M$, the above sum is a
partial Euler characteristic (computed from the back of the resolution)
that computes the rank of the $k$th syzygy module of $M$.  In particular,
each such linear functional is strictly nonnegative when evaluated on
the Betti table of a finite length module, and hence we have
$B_\QQ^\text{fin}(R)\subseteq C$.

The reverse containment $C\subseteq B_\QQ^\text{fin}(R)$ requires a
limiting argument.  We show that for each $i$, there is a sequence of pairs
of positive scalars and modules $\{(\lambda_{i,j}, M_{i,j})\}_{j=1}^\infty$
such that
\[
\rho_i = \lim_{j\to\infty} \lambda_{i,j} \beta^R(M_{i,j}).
\]
The key fact used in the construction of these $R$-modules is that there
exist local ring analogues to the $S$-modules with pure resolutions
constructed in Theorem~\ref{thm:exist std}.  Thus, given a degree
sequence $d\in \ZZ^{n+1}$, we may construct an $R$-module $M(d)$ whose
total Betti numbers are computed (up to scalar multiple) by the
Herzog--K\"uhl equations.    The precise existence statement for the
$R$-module $M(d)$ is given in Lemma~\ref{lem:to local} below.

As noted in Remark~\ref{rmk:HK}, the Betti table of $M(d)$ is, up to scalar
multiple, given by
\begin{align*}
\mathfrak b(d) 
&=\prod_{\ell\ne i}
|d_\ell-d_i|\left(\frac{1}{\prod_{\ell\ne 0} |d_\ell-d_0|},
\frac{1}{\prod_{\ell\ne 1} |d_\ell-d_1|}, \dots, \frac{1}{\prod_{\ell\ne n}
|d_\ell-d_n|}\right)\\
&=\left(\frac{\prod_{\ell\ne i}
|d_\ell-d_i|}{\prod_{\ell\ne 0} |d_\ell-d_0|}, \frac{\prod_{\ell\ne i}
|d_\ell-d_i|}{\prod_{\ell\ne 1} |d_\ell-d_1|}, \dots, \frac{\prod_{\ell\ne
i}
|d_\ell-d_i|}{\prod_{\ell\ne n} |d_\ell-d_n|}\right)\in \QQ^{n+1}.
\end{align*}
Note that $\mathfrak b(d)_i=1$. 
By carefully choosing degree sequences $d^{i,j}$, we will
realize $\rho_i$ as the desired limit using $M_{i,j}=M(d^{i,j})$.  To make this choice, set 
$d^{i,j}:=(0, j, 2j, \ldots, ij, ij+1,(i+1)j+1, \ldots, (n-1)j+1)$, 
so that 
\[
d^{i,j}_k=\begin{cases}
kj &\text{if $k\leq i$,}\\
(k-1)j+1 &\text{if $k>i$.} 
\end{cases}
\]
For these degree sequences, the Herzog--K\"uhl equations imply that, as
$i\to \infty$, the $i$th and $(i+1)$st Betti numbers go to infinity more
quickly than the other Betti numbers do.  Of course, this limit does not
make sense for graded Betti numbers.
In the local case, where the Betti numbers are ungraded, we may consider
such limits.

We thus set $M_{i,j}:=M(d^{i,j})$ and
$\lambda_{i,j}:=\frac{1}{\beta^R_i(M(d^{i,j}))}$. This yields
\[
\lambda_{i,j}\beta^R(M_{i,j})=\mathfrak b(d^{i,j}),
\]
since they are equal up to scalar multiple and the $i$th entry in both
vectors is equal to $1$. 

We now claim that $\lim_{j\to \infty} \mathfrak b(d^{i,j})=\rho_i$.  By construction,
the limit equals $1$ in the $i$th position.  Also, each element
$\mathfrak b(d^{i,j})$ lies in the linear subspace given by $\sum_{k=0}^n
(-1)^k\beta_k=0$. Thus it suffices to show that $\lim_{j\to \infty}
\mathfrak b(d^{i,j})_k=0$ for $k\ne i,i+1$, which we directly compute:
\begin{align*}
  \lim_{j\to \infty} \mathfrak b(d^{i,j})_k
  &=\lim_{j\to\infty}
    \frac{\prod_{\ell \ne j} |d^{i,j}_\ell-d^{i,j}_i |}{\prod_{\ell \ne k}
    |d^{i,j}_\ell-d^{i,j}_k|}
  =\lim_{j\to\infty}\frac{O(j^{n-1})}{O(j^{n})} =0.
\end{align*}
Thus $B_\QQ^\text{fin}(R)$ contains points that are arbitrarily close to each $\rho_i$.  Since $C$ equals the interior of the closed cone spanned by the $\rho_i$, we have shown that $B_\QQ^\text{fin}(R)$ contains $C$, as desired.
\end{proof}

The following lemma is proven in \cite{beks-local}*{Proposition~2.1}.
\begin{lemma}\label{lem:to local}
Let $R$ be an $n$-dimensional regular local ring, and let $d=(d_0, \dots,
d_n)$ be a degree sequence.  If $N$ is the cokernel of the pure resolution
of type $d$ constructed in Theorem~\ref{thm:exist std}, then there exists a
finite length $R$-module $M(d)$ where $\beta_i^R(M(d))=\beta_{i,d_i}(N)$.
\end{lemma}

\section{Extremal Betti Tables in the Multigraded case}
\label{sec:multi}
Whereas in the previous section, we considered regular local rings, we now
move in the opposite direction by refining the grading on the polynomial
ring.  As we will see, this greatly increases the complexity of the
situation. The results discussed in this section stem from our original
work plus extended discussions with Eisenbud and Schreyer.

We restrict attention to the simplest example of a finely graded polynomial ring, namely $T:=\kk[x,y]$ with the
bigrading $\deg(x)=(1,0)$ and $\deg(y)=(0,1)$.  We seek
$T$-modules $M$ of finite length such that $\beta^T(M)$ is extremal. 

Over $S$, extremal was synonymous with having a pure resolution, but over
$T$ this is not the case.  In fact, there cannot exist a finite length
module $M$ with a resolution of $\FF$ where each $\FF_i$ is generated in a
single bidegree.  This is because $T$ is finely graded, so the
cokernel of any map 
$
T^a(-\mu_1,-\mu_2)\, \longleftarrow \, T^b(-\lambda_1,-\lambda_2)
$ 
has codimension at most $1$.

There are, however, other natural candidates for extremal Betti tables.
For instance, in the standard $\ZZ$-graded case, every pure resolution over
$\kk[x,y]$ can be a realized by taking the resolution of a quotient of
monomial ideals~\cite{boij-sod1}*{Remark~3.2}.  Since each of these modules
is naturally bigraded, we might expect that these provide extremal Betti
tables in the bigraded sense, as well as in the graded sense. While this is
quite often the case (see Example~\ref{ex:monomials}), there are many other
extremal bigraded Betti tables as well.

To describe a sufficient condition for extremality, we introduce the notion
of the \defi{matching graph} $\Gamma(M)$ of a bigraded $T$-module of finite
length. By imposing rather weak conditions on matching graphs, we produce a
wide array of bigraded $T$-modules with extremal Betti tables. This
illustrates the additional complexity that arises from refined gradings.

\begin{sufficient condition}\label{cond:3}
Let $M$ be a bigraded $T$-module of finite length.  If its matching graph
$\Gamma(M)$ is $(1,1)$-valent and connected, then $\beta^T(M)$ is extremal.
\end{sufficient condition}

For a bigraded $T$-module $M$ of finite length, let $\FF$ be the bigraded
minimal free resolution of $M$. The \defi{matching graph} of $M$ is a graph
whose vertices have weights in $\ZZ$ and whose edges are of two types:
$x$-edges and $y$-edges.  The vertices correspond to the degrees of the
generators of the $\FF_i$; to a vertex $\alpha \in \ZZ^2$, we assign the
weight $\beta^T_{0,\alpha}(M)+\beta^T_{1,\alpha}(M)+\beta^T_{2,\alpha}(M)$.
We then include an $x$-edge (or $y$-edge, respectively) between any two
vertices with the same $x$-degree (or $y$-degree). 

If a vertex of $\Gamma(M)$ meets precisely $a$ of the $x$-edges and
precisely $b$ of the $y$-edges, then we say that this vertex has
\defi{valency} $(a,b)$.  If all of the vertices of $\Gamma(M)$ have valency
$(a,b)$, then we say that $\Gamma(M)$ is an $(a,b)$-valent graph.  In
addition, we say that $\Gamma(M)$ is \defi{connected} if the underlying
graph (i.e., the graph on the same vertices whose edges are the union of
the $x$-edges and $y$-edges of $\Gamma(M)$) is connected.

\begin{example}
Let $M=T/\<x^2,xy,y^2\>$.  The minimal free resolution of $M$ has the form 
\[
T^1\longleftarrow \begin{matrix} T^1(-2,0)\\ \oplus \\ T^1(-1,-1) \\ \oplus
\\ T^1(0,-2) \end{matrix} \longleftarrow \begin{matrix} T^1(-2,-1)\\ \oplus
\\ T^1(-1,-2)  \end{matrix}\longleftarrow 0.
\]
Using the natural embedding of the matching graph $\Gamma(M)$ in the first
orthant, $\Gamma(M)$ has $x$-edges as shown in the figure on the left. 
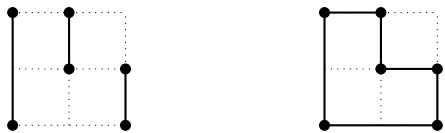
\[
\begin{tikzpicture}[scale=.75]
\filldraw (0,0) circle (2.5pt);
\filldraw (2,0) circle (2.5pt);
\filldraw (1,1) circle (2.5pt);
\filldraw (0,2) circle (2.5pt);
\filldraw (2,1) circle (2.5pt);
\filldraw (1,2) circle (2.5pt);
\draw[thin,dotted] (1,0)--(1,2);
\draw[thin,dotted] (0,0)--(0,2);
\draw[thin,dotted] (2,0)--(2,2);
\draw[thin,dotted] (0,0)--(2,0);
\draw[thin,dotted] (0,1)--(2,1);
\draw[thin,dotted] (0,2)--(2,2);
\draw[thick] (0,0)--(0,2);
\draw[thick] (2,0)--(2,1);
\draw[thick] (1,1)--(1,2);
\end{tikzpicture}
\hspace{2.5cm}
\begin{tikzpicture}[scale=.75]
\filldraw (0,0) circle (2.5pt);
\filldraw (2,0) circle (2.5pt);
\filldraw (1,1) circle (2.5pt);
\filldraw (0,2) circle (2.5pt);
\filldraw (2,1) circle (2.5pt);
\filldraw (1,2) circle (2.5pt);
\draw[thin,dotted] (1,0)--(1,2);
\draw[thin,dotted] (0,0)--(0,2);
\draw[thin,dotted] (2,0)--(2,2);
\draw[thin,dotted] (0,0)--(2,0);
\draw[thin,dotted] (0,1)--(2,1);
\draw[thin,dotted] (0,2)--(2,2);
\draw[thick] (0,0)--(0,2)--(1,2) -- (1,1) --(2,1) -- (2,0) -- (0,0);
\end{tikzpicture}
\]
We omit the weights on the vertices, since all weights are $1$.  
The graph $\Gamma(M)$ appears on the right, and is $(1,1)$-valent and connected.  Hence $\beta^T(M)$ is
extremal by Claim~\ref{cond:3}. 
\end{example}
\begin{example}
Let $M=\langle x,y\rangle / \langle x^2,xy^2,y^3\rangle$.  Then $\Gamma(M)$ fails to be $(1,1)$-valent.  In fact, at each vertex of the form $(1,*)$, there are $3$ $x$-edges.  In this case,
$
\beta^T(M)$ equals $\beta^T(\langle x\rangle / \langle x^2,xy^2 \rangle) + \beta^T(\langle y\rangle /\langle xy,y^3 \rangle).
$
These last two Betti tables are extremal by Claim~\ref{cond:3}.
\end{example}

\begin{proof}[Proof of Claim~\ref{cond:3}]
For any $\lambda_1\in \ZZ$, we can consider the subgraph of $\Gamma(M)$
obtained by restricting to the vertices of $\Gamma(M)$ whose degrees have
the form $(\lambda_1,*)$.  By definition of the $x$-edges, there will be an
$x$-edge between any two vertices of this subgraph.  Hence, by the
$(1,1)$-valency, we see that $\Gamma(M)$ has at most two vertices of the
form $(\lambda_1,*)$.

In fact, for each $\lambda_1\in \ZZ$, we claim that $\Gamma(M)$ has either
zero or two vertices of the form $(\lambda_1,*)$.  The bigraded Hilbert
series of $M$ is given by the rational function
\[
H_M(s_1,s_2)=\frac{K_M(s_1,s_2)}{(1-s_1)(1-s_2)}:=\frac{\sum_{i=0}^2\sum_{\lambda
\in \ZZ^2}(-1)^i\beta^T_{i,\lambda}(M)\mathbf{s}^\lambda}{(1-s_1)(1-s_2)}.
\]
Since $M$ has finite length, $H_M(s_1,s_2)$ is actually a polynomial.  This
implies that the $K$-polynomial of $M$, $K_M$, is in $\<1-s_1\>
\cap\<1-s_2\>\subseteq\ZZ[s_1,s_2]$.
We thus have
\[
K_M(s_1,1)=\sum_{\lambda_1\in\ZZ} \left(\sum_{\lambda_2\in\ZZ}
\beta^T_{0,(\lambda_1,\lambda_2)}-
\beta^T_{1,(\lambda_1,\lambda_2)}
+\beta^T_{2,(\lambda_1,\lambda_2)}  \right)s_1^{\lambda_1}  =0.
\]
Thus, if $\Gamma(M)$ has a vertex of the form $(\lambda_1,*)$,
then it has at least two such vertices of this form.  Further, one
of these vertices must correspond to a generator of $\FF_1$ and the other must
correspond to a generator of either $\FF_0$ or $\FF_2$, and the corresponding
Betti numbers be equal.  By alternately considering Betti numbers with the same $x$-degrees and Betti numbers with the same $y$-degrees, we may show that any two Betti numbers in the same connected component of $\Gamma(M)$ must have the value.  The connectedness of $\Gamma(M)$ ths implies that each
nonzero Betti number of $M$ has the same positive value.

Suppose now that $\beta^T(M)=a'\beta^T(M')+a''\beta^T(M'')$ for some
bigraded modules $M',M''$ of finite length and some $a',a''\in \mathbb
Q_{>0}$.  We start by considering a bidegree $(\lambda_1,\lambda_2)$ where
$\beta^T_{0,(\lambda_1,\lambda_2)}(M')=r$.  Since $K_{M'}(s_1,1)=0$ and
$\Gamma(M)$ is $(1,1)$-valent, the argument above implies that there is a
unique $\mu_2$ such that $\beta^T_{1,(\lambda_1,\mu_2)}(M')\ne 0$, and
hence this Betti number must also equal $r$.  We then consider $y$-degrees,
and a similar argument shows that there is a unique $\mu_1$ such that
either (but not both) $\beta^T_{0,(\mu_1,\mu_2)}(M')\ne 0$ or
$\beta^T_{2,(\mu_1,\mu_2)}(M')\ne 0$. In either case, this Betti number
must also equal $r$.  

Continuing to alternate between $x$-degrees and $y$-degrees, we eventually
form a subcycle of $\Gamma(M)$.  However, since $\Gamma(M)$ is
$(1,1)$-valent and connected, this cycle must equal $\Gamma(M)$, so we have
shown that $\beta^T(M')$ is simply $r$ times $\beta^T(M)$.
\end{proof}

\begin{example}\label{ex:monomials}
Quotients of monomial ideals provide many examples of extremal bigraded
Betti tables.  For instance, let $M=I/J$, where $I=\<x^4,xy^2,x^2y,y^4\>$
and $J=\<x^6,x^3y^3,y^6\>$.  Then $\Gamma(M)$ is the graph on the left
(each vertex has weight $1$),
which is extremal by Claim~\ref{cond:3}.  
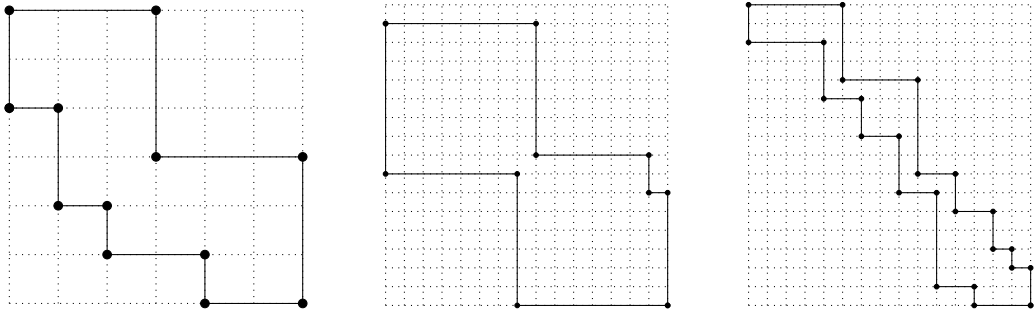
\[
\begin{tikzpicture}[scale=.65]
\filldraw (4,0) circle (2.5pt);
\filldraw (2,1) circle (2.5pt);
\filldraw (1,2) circle (2.5pt);
\filldraw (0,4) circle (2.5pt);
\filldraw (2,2) circle (2.5pt);
\filldraw (3,3) circle (2.5pt);
\filldraw (4,1) circle (2.5pt);
\filldraw (1,4) circle (2.5pt);
\filldraw (6,0) circle (2.5pt);
\filldraw (0,6) circle (2.5pt);
\filldraw (6,3) circle (2.5pt);
\filldraw (3,6) circle (2.5pt);
\filldraw (0,4) circle (2.5pt);
\draw[thin,dotted] (0,0)--(6,0);
\draw[thin,dotted] (0,1)--(6,1);
\draw[thin,dotted] (0,2)--(6,2);
\draw[thin,dotted] (0,3)--(6,3);
\draw[thin,dotted] (0,4)--(6,4);
\draw[thin,dotted] (0,5)--(6,5);
\draw[thin,dotted] (0,6)--(6,6);
\draw[thin,dotted] (0,0)--(0,6);
\draw[thin,dotted] (1,0)--(1,6);
\draw[thin,dotted] (2,0)--(2,6);
\draw[thin,dotted] (3,0)--(3,6);
\draw[thin,dotted] (4,0)--(4,6);
\draw[thin,dotted] (5,0)--(5,6);
\draw[thin,dotted] (6,0)--(6,6);
\draw[thin]
(4,0)--(4,1)--(2,1)--(2,2)--(1,2)--(1,4)--(0,4)--(0,6)--(3,6)--(3,3)--(6,3)--(6,0)--(4,0);
\end{tikzpicture}
\hspace{1cm}
\begin{tikzpicture}[scale=.25]
\filldraw (0,15) circle (3.5pt);
\filldraw (8,15) circle (3.5pt);
\filldraw (8,8) circle (3.5pt);
\filldraw (14,8) circle (3.5pt);
\filldraw (14,6) circle (3.5pt);
\filldraw (15,6) circle (3.5pt);
\filldraw (15,0) circle (3.5pt);
\filldraw (0,15) circle (3.5pt);
\filldraw (0,7) circle (3.5pt);
\filldraw (7,7) circle (3.5pt);
\filldraw (7,0) circle (3.5pt);
\filldraw (15,0) circle (3.5pt);
\draw[thin](15,0)--(7,0)--(7,7)
--(0,7)--(0,15)--(8,15)--(8,8)--(14,8)--(14,6)--(15,6)--(15,0);
\draw[thin,dotted] (0,0)--(15,0);
\draw[thin,dotted] (0,1)--(15,1);
\draw[thin,dotted] (0,2)--(15,2);
\draw[thin,dotted] (0,3)--(15,3);
\draw[thin,dotted] (0,4)--(15,4);
\draw[thin,dotted] (0,5)--(15,5);
\draw[thin,dotted] (0,6)--(15,6);
\draw[thin,dotted] (0,7)--(15,7);
\draw[thin,dotted] (0,8)--(15,8);
\draw[thin,dotted] (0,9)--(15,9);
\draw[thin,dotted] (0,10)--(15,10);
\draw[thin,dotted] (0,11)--(15,11);
\draw[thin,dotted] (0,12)--(15,12);
\draw[thin,dotted] (0,13)--(15,13);
\draw[thin,dotted] (0,14)--(15,14);
\draw[thin,dotted] (0,15)--(15,15);
\draw[thin,dotted] (0,16)--(15,16);
\draw[thin,dotted] (0,0)--(0,16);
\draw[thin,dotted] (1,0)--(1,16);
\draw[thin,dotted] (2,0)--(2,16);
\draw[thin,dotted] (3,0)--(3,16);
\draw[thin,dotted] (4,0)--(4,16);
\draw[thin,dotted] (5,0)--(5,16);
\draw[thin,dotted] (6,0)--(6,16);
\draw[thin,dotted] (7,0)--(7,16);
\draw[thin,dotted] (8,0)--(8,16);
\draw[thin,dotted] (9,0)--(9,16);
\draw[thin,dotted] (10,0)--(10,16);
\draw[thin,dotted] (11,0)--(11,16);
\draw[thin,dotted] (12,0)--(12,16);
\draw[thin,dotted] (13,0)--(13,16);
\draw[thin,dotted] (14,0)--(14,16);
\draw[thin,dotted] (15,0)--(15,16);
\end{tikzpicture}
\hspace{1cm}
\begin{tikzpicture}[scale=.25]
\filldraw (15,0) circle (3.5pt);
\filldraw (15,2) circle (3.5pt);
\filldraw (14,2) circle (3.5pt);
\filldraw (14,3) circle (3.5pt);
\filldraw (13,3) circle (3.5pt);
\filldraw (13,5) circle (3.5pt);
\filldraw (11,5) circle (3.5pt);
\filldraw (11,7) circle (3.5pt);
\filldraw (9,7) circle (3.5pt);
\filldraw (9,12) circle (3.5pt);
\filldraw (5,12) circle (3.5pt);
\filldraw (5,16) circle (3.5pt);
\filldraw (0,16) circle (3.5pt);
\filldraw (15,0) circle (3.5pt);
\filldraw (12,0) circle (3.5pt);
\filldraw (12,1) circle (3.5pt);
\filldraw (10,1) circle (3.5pt);
\filldraw (10,6) circle (3.5pt);
\filldraw (8,6) circle (3.5pt);
\filldraw (8,9) circle (3.5pt);
\filldraw (6,9) circle (3.5pt);
\filldraw (6,11) circle (3.5pt);
\filldraw (4,11) circle (3.5pt);
\filldraw (4,14) circle (3.5pt);
\filldraw (0,14) circle (3.5pt);
\draw[thin](15,0)--(15,2)--(14,2)
--(14,3)--(13,3)--(13,5)--(11,5)--(11,7)--(9,7)--(9,12)--(5,12)--(5,16)--(0,16);
\draw[thin](15,0)--(12,0)--(12,1)--(10,1)--(10,6)--(8,6)
--(8,9)--(6,9)--(6,11)--(4,11)--(4,14)--(0,14)--(0,16);
\draw[thin,dotted] (0,0)--(15,0);
\draw[thin,dotted] (0,1)--(15,1);
\draw[thin,dotted] (0,2)--(15,2);
\draw[thin,dotted] (0,3)--(15,3);
\draw[thin,dotted] (0,4)--(15,4);
\draw[thin,dotted] (0,5)--(15,5);
\draw[thin,dotted] (0,6)--(15,6);
\draw[thin,dotted] (0,7)--(15,7);
\draw[thin,dotted] (0,8)--(15,8);
\draw[thin,dotted] (0,9)--(15,9);
\draw[thin,dotted] (0,10)--(15,10);
\draw[thin,dotted] (0,11)--(15,11);
\draw[thin,dotted] (0,12)--(15,12);
\draw[thin,dotted] (0,13)--(15,13);
\draw[thin,dotted] (0,14)--(15,14);
\draw[thin,dotted] (0,15)--(15,15);
\draw[thin,dotted] (0,16)--(15,16);
\draw[thin,dotted] (0,0)--(0,16);
\draw[thin,dotted] (1,0)--(1,16);
\draw[thin,dotted] (2,0)--(2,16);
\draw[thin,dotted] (3,0)--(3,16);
\draw[thin,dotted] (4,0)--(4,16);
\draw[thin,dotted] (5,0)--(5,16);
\draw[thin,dotted] (6,0)--(6,16);
\draw[thin,dotted] (7,0)--(7,16);
\draw[thin,dotted] (8,0)--(8,16);
\draw[thin,dotted] (9,0)--(9,16);
\draw[thin,dotted] (10,0)--(10,16);
\draw[thin,dotted] (11,0)--(11,16);
\draw[thin,dotted] (12,0)--(12,16);
\draw[thin,dotted] (13,0)--(13,16);
\draw[thin,dotted] (14,0)--(14,16);
\draw[thin,dotted] (15,0)--(15,16);
\end{tikzpicture}
\]
The middle and right graphs above correspond to the matching graphs of
other quotients of monomial ideals: the middle graph is the matching graph of $\langle x^7,y^7\rangle / \langle x^{15}, x^{14}y^6,x^8y^8,y^{15} \rangle$.
These types of examples
may be very far from the pure resolutions we saw in \S\ref{sec:graded}.  For
instance, we can produce an extremal bigraded Betti table given by a
resolution $\FF$, where $\FF_i$ has minimal generators in arbitrarily many
different bidegrees.

Note that these examples are not pure with respect to the $\ZZ$-grading.
\end{example}
Claim~\ref{cond:3} begs the question of which $(1,1)$-valent, connected
graphs can be realized as $\Gamma(M)$ for some $M$. If $\Gamma(M)$ is a
$(1,1)$-valent, connected graph that comes from a quotient of monomial
ideals, then it decomposes as the union of two nonintersecting monotonic
paths (from the upper left corner to the lower right corner).  But the
following example illustrates that not all extremal Betti tables arise in
this way.

\begin{example}\label{ex:heartshape}
The cokernel of the matrix below induces the following matching graph:
\[
\begin{tabular}{ccc}
\parbox[h]{4cm}{
\bordermatrix{&\binom{3}{0}&\binom{2}{1}&\binom{1}{2}&\binom{0}{3}\cr
                \binom{1}{0}&x^2 &  xy & y^2 & 0\cr
                \binom{0}{1}& 0 & x^2  & xy& y^2
                }
                \vspace*{.5cm}
}
&
\hspace*{1.5cm}
&
\parbox[h]{3cm}
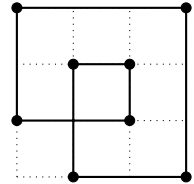
{
\begin{tikzpicture}[scale=.75]
\filldraw (1,0) circle (2.5pt);
\filldraw (0,1) circle (2.5pt);
\filldraw (3,0) circle (2.5pt);
\filldraw (2,1) circle (2.5pt);
\filldraw (1,2) circle (2.5pt);
\filldraw (0,3) circle (2.5pt);
\filldraw (2,2) circle (2.5pt);
\filldraw (3,3) circle (2.5pt);
\draw[thin,dotted] (0,0)--(3,0);
\draw[thin,dotted] (0,1)--(3,1);
\draw[thin,dotted] (0,2)--(3,2);
\draw[thin,dotted] (0,3)--(3,3);
\draw[thin,dotted] (0,0)--(0,3);
\draw[thin,dotted] (1,0)--(1,3);
\draw[thin,dotted] (2,0)--(2,3);
\draw[thin,dotted] (3,0)--(3,3);
\draw[thick] (1,0)--(3,0)--(3,3)--(0,3)--(0,1)--(2,1)--(2,2)--(1,2)--(1,0);
\end{tikzpicture}
}
\end{tabular}
\]
\end{example}

However, not every $(1,1)$-valent connected graph arises as the
matching graph of a module.

\begin{example}\label{ex:pacman}
Suppose that the graph on the left below is the matching
graph of a module $M$ of finite length. Then the free resolution of $M$
has the form shown on the right. 

\begin{tabular}{cc}
\parbox[h]{2cm}{
\[
\hspace*{-12.5cm}
\begin{tikzpicture}[scale=.75]
\filldraw (0,0) circle (2.5pt);
\filldraw (1,1) circle (2.5pt);
\filldraw (3,0) circle (2.5pt);
\filldraw (2,1) circle (2.5pt);
\filldraw (1,3) circle (2.5pt);
\filldraw (0,2) circle (2.5pt);
\filldraw (2,2) circle (2.5pt);
\filldraw (3,3) circle (2.5pt);
\draw[thin,dotted] (0,0)--(3,0);
\draw[thin,dotted] (0,1)--(3,1);
\draw[thin,dotted] (0,2)--(3,2);
\draw[thin,dotted] (0,3)--(3,3);
\draw[thin,dotted] (0,0)--(0,3);
\draw[thin,dotted] (1,0)--(1,3);
\draw[thin,dotted] (2,0)--(2,3);
\draw[thin,dotted] (3,0)--(3,3);
\draw[thick] (0,0)--(3,0)--(3,3)--(1,3)--(1,1)--(2,1)--(2,2)--(0,2)--(0,0);
\end{tikzpicture}
\]
}
&
\parbox[h]{7cm}{
\[
\hspace*{-2.75cm}
\begin{matrix}
T\\ \oplus \\ T(-1,-1)
\end{matrix}
\overset{\phi}{\xleftarrow{\hspace{1.15cm}}}
\begin{matrix}
T(-3,0)\\ \oplus \\ T(-2,-1)\\ \oplus \\T(-1,-3)\\ \oplus\\ T(0,-2)
\end{matrix}
\overset{\psi}{\xleftarrow{\hspace{1.15cm}}}
\begin{matrix}
T(-2,2)\\ \oplus \\ T(-3,-3)
\end{matrix}
\longleftarrow 0
\]
}
\end{tabular}

\noindent In this case, the matrix $\phi$ would have the form:
\[
\phi=\bordermatrix{&\binom{3}{0}&\binom{2}{1}&\binom{1}{3}&\binom{0}{2}\cr
                \binom{0}{0}&x^3 &  a_1x^2y & a_2xy^3 & y^2\cr
                \binom{1}{1}& 0 & -x  & y^2& 0
                }
\]
for some scalars $a_1, a_2$.  After performing an appropriate row operation
and column operation, we can assume that $a_1=a_2=0$.   However, the kernel
of the resulting matrix is generated by $T(-3,-2)\oplus T(-2,-3)$,
providing the contradiction.
\end{example}

Though we know of no condition for determining which $(1,1)$-valent,
connected graphs arise as $\Gamma(M)$ for some $M$,
Claim~\ref{cond:3} provides a zoo of extremal rays. If we
restrict to Betti tables whose support is contained in
the square with corners $(0,0)$ and $(3,3)$, Claim~\ref{cond:3} produces 74
extremal rays. They are generated by the tables of quotients of monomial
ideals, along with the table in Example~\ref{ex:heartshape} and its dual. 
We conclude with a conjecture.

\begin{conj}
All extremal Betti tables of the cone of bigraded $T$-modules with finite
length are generated by Betti tables of modules $M$ that
satisfy Claim~\ref{cond:3}.
\end{conj}


\end{document}